\documentclass[a4paper,12pt]{article}
\usepackage{amsmath}
\usepackage{amssymb}
\usepackage{amsthm}
\usepackage[latin1]{inputenc}
\usepackage[T1]{fontenc}
\usepackage{wasysym}
\usepackage{graphicx}
\usepackage{cleveref}
\usepackage{enumerate}

\newtheorem{thm}{Theorem}[section]
\newtheorem*{conj}{Conjecture}
\newtheorem{lem}[thm]{Lemma}
\newtheorem{prop}[thm]{Proposition}

\crefname{lem}{Lemma}{Lemmas}
\crefname{thm}{Theorem}{Theorems}
\crefname{prop}{Proposition}{Propositions}
\crefname{cor}{Corollary}{Corollaries}
\crefname{fig}{Figure}{Figures}

\begin{document}
\title{Higher connectivity of fiber graphs{\newline}of Gr\"{o}bner bases}
\author{Samu Potka}
\date{}
\maketitle

\begin{abstract}
Fiber graphs of Gr\"{o}bner bases from contingency tables are important in statistical hypothesis testing, where one studies random walks on these graphs using the Metropolis-Hastings algorithm. The connectivity of the graphs has implications on how fast the algorithm converges. In this paper, we study a class of fiber graphs with elementary combinatorial techniques and provide results that support a recent conjecture of Engstr\"{o}m: the connectivity is given by the minimum vertex degree.
\end{abstract}

\numberwithin{figure}{section}

\section{Introduction}

We will study a class of graphs coming from Gr\"{o}bner bases related to the two-way $n{\times}n$ contingency tables with equal row and column sums. By summing the entries of the tables both row-wise and column-wise, it is easy to see that the $n{\times}n$ tables are the only ones that can satisfy this property. Let $G(n,r)$ be a graph whose vertices are the $n{\times}n$-matrices of non-negative integers with all row and column sums $r$. Two vertices are adjacent if one can move between the corresponding matrices by adding one to two entries and subtracting one from two others. As an example, consider the graph $G(3,2)$, drawn in \cref{fig:example}. The vertices are the $3{\times}3$-matrices of non-negative integers with row and column sums two. The graph $G(n,r)$ is the underlying undirected graph of a \emph{fiber graph} of a reduced Gr\"{o}bner basis and the edges correspond to \emph{Markov moves}. After stating our main result, we shortly review the basics of algebraic statistics.

\begin{figure}[htbp!]
\centering
\includegraphics[scale=0.57]{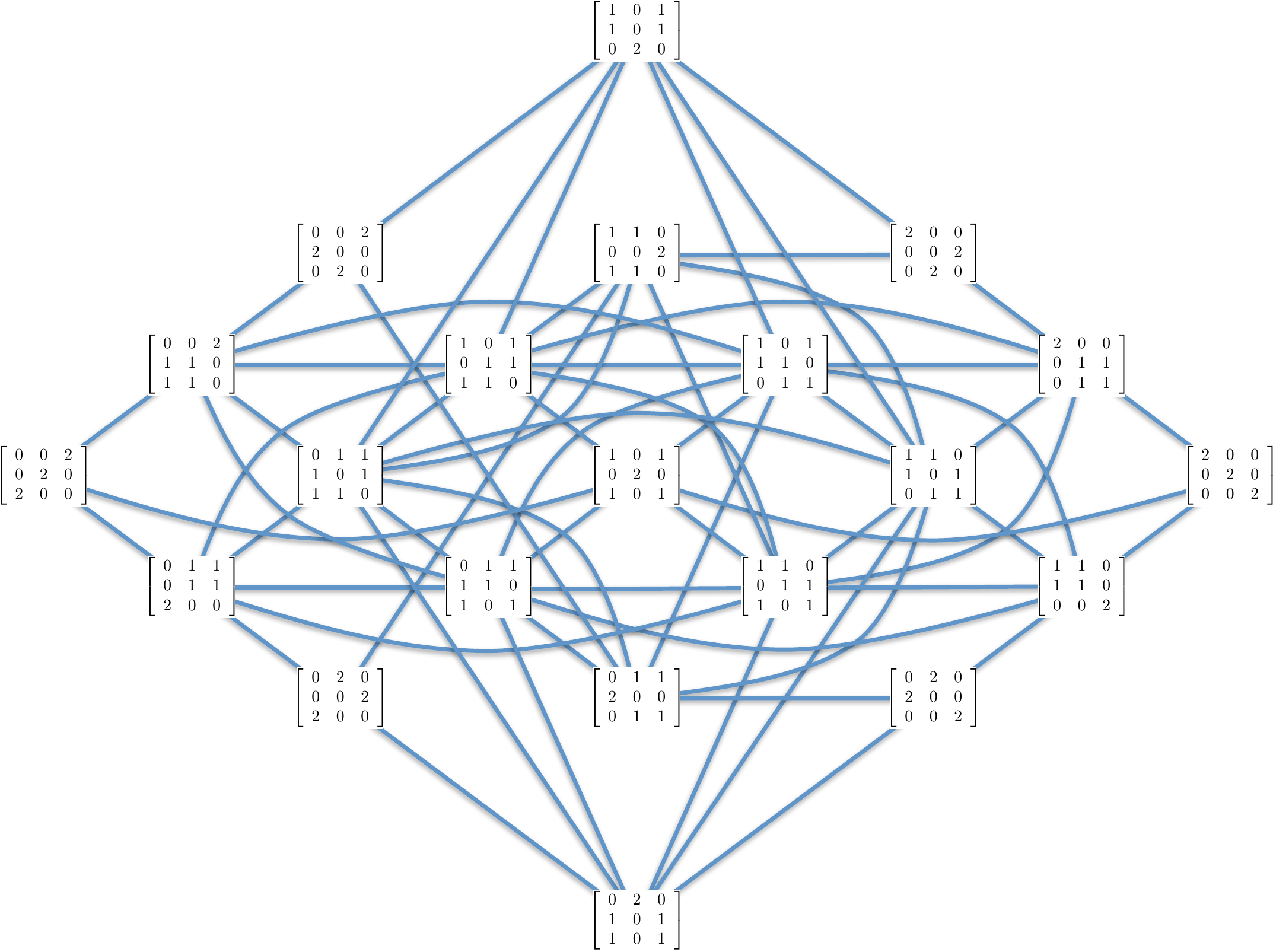}
\caption{The graph $G(3,2)$.}
\label[fig]{fig:example}
\end{figure}

To state our main result, we need to mention some standard definitions from graph theory. The \emph{degree} $d(v)$ of a vertex $v$ in $G$ is the number of edges at $v$. The \emph{minimum degree} $\delta(G)$ of a graph $G$ is the smallest of the degrees in the graph. A graph $G$ is $k$\emph{-connected}, $k \in \mathbb{N}$, if $|G| > k$ and $G - X$ is connected for every set $X \subseteq V(G)$ with $|X| < k$. The \emph{connectivity} ${\kappa}(G)$ of a graph $G$ is the largest $k$ such that $G$ is $k$-connected. 

The \emph{Metropolis-Hastings algorithm} can be used for statistical tests for contingency tables. The algorithm performs a random walk on the fiber graph containing the contingency table we want to study \cite{diaconis-sturmfels}. The connectivity of the fiber graphs affects the convergence of the algorithm: typically, the lower the connectivity, the slower the convergence \cite{hlw}. Our main result is: 
\newtheorem*{thm:maintheorem}{Theorem \ref{thm:maintheorem}}
\begin{thm:maintheorem}
The connectivity $\kappa(G(n,r)) = {n \choose 2}$ for $r > 2$.
\end{thm:maintheorem}
We also prove several other statements regarding $G(n,r)$. The proof of the main result is based on Liu's criterion \cite{liu}, proved, for example, in \cite{bjorner-vorwerk}:
\newtheorem*{lem:liuc}{Lemma \ref{lem:liuc}}
\begin{lem:liuc}[Liu's criterion]
Let $G$ be a connected graph and
$|V(G)| > k$. If for any two vertices $u$ and $v$ of $G$ with distance
$d_G(u,v) = 2$ there are $k$ disjoint $u-v$ paths in $G$, then $G$ is $k$-connected.
\end{lem:liuc}
\newtheoremstyle{component}{}{}{}{}{\itshape}{.}{.5em}{\thmnote{#3}#1}
\theoremstyle{component}
\newtheorem*{component}{}
For the first time, the following conjecture is confirmed for a large class of fiber graphs of an important and common class of Gr\"{o}bner bases.
\begin{conj}[Engstr\"{o}m '12, \cite{e}\cite{p}]
The connectivity of a large fiber graph of a
reduced Gr\"obner basis of a lattice ideal is given by the minimum vertex degree
of the fiber graph.
\end{conj}

The technical version of the conjecture with the condition of a large fiber graph is spelled out in the appendix.

\subsection{The basics of algebraic statistics}

Let us review the basics of algebraic statistics. See the foundational paper \cite{diaconis-sturmfels} or the textbook \cite{dss} for an introduction to the field. 

Fix an integer matrix $A \in \mathbb{Z}^{d \times n}$ whose column sums are equal. The \emph{probability simplex} is $\Delta_{n-1} = \{p \in [0, 1]^n : \sum_{i=1}^n p_i = 1\}$. Let $\mathcal{M}_A = \{p \in \Delta_{n-1} : \text{log } p \in \text{rowspan}(A) \}$ be the \emph{log-linear model} associated with the matrix $A$. The vector $Au$ is the \emph{minimal sufficient statistic} for $\mathcal{M}_A$ and $\mathcal{F}(u) = \{v \in \mathbb{N}^n : Av = Au \}$ the \emph{fiber} of a contingency table $u$, represented in a vectorized form. Let ker$_{\mathbb{Z}}(A)$ be the integer kernel of the matrix $A$. The finite set $\mathcal{B} \subset $ ker$_{\mathbb{Z}}(A)$ is a \emph{Markov basis} for $\mathcal{M}_A$ if there exists a sequence $u_1, ..., u_L \in \mathcal{B}$ such that $v' = v + \sum_{k = 1}^L u_k$ and $v + \sum_{k = 1}^L u_k \geq 0$ for all $l = 1, ..., L$; all contingency tables $u$ and all pairs $v, v' \in \mathcal{F}(u)$. The elements of the Markov basis are called \emph{Markov moves}.

Another way of describing Markov bases is via finite subsets of lattices. In this case, we are interested in the integer lattice ker$_{\mathbb{Z}}(A)$, where $A$ is the matrix associated with the log-linear model. The \emph{fiber} $\mathcal{F}(u)$ of $u \in \mathbb{N}^n$, for example, a contingency table in vectorized form, is the set $\{v \in \mathbb{N}^n : u - v \in \mathcal{L} \}$, where $\mathcal{L}$ is a lattice. Note that if $\mathcal{L} = $ ker$_{\mathbb{Z}}(A)$, this definition is exactly the same as the definition of the fiber of a contingency table mentioned earlier. Let $\mathcal{B}$ be an arbitrary finite subset of $\mathcal{L}$. The subset determines an undirected graph $\mathcal{F}(u)_{\mathcal{B}}$ whose vertices are the elements of $\mathcal{F}(u)$. Two vertices $v$ and $v'$ are connected by an edge if either $v - v'$ or $v' - v$ is in $\mathcal{B}$. The subset $\mathcal{B}$ is a \emph{Markov basis} of $\mathcal{L}$ if the graphs $\mathcal{F}(u)_{\mathcal{B}}$ are connected for all $u \in \mathbb{N}^n$. Fix a weight vector $w \in \mathbb{R}^n$ such that $b \cdot w < 0$ for all $b \in \mathcal{B}$. The graph $\mathcal{F}(u)_{\mathcal{B}}$ is an acyclic directed graph if the edges are now directed: $v \rightarrow v'$, and present whenever $v' - v$ is in $\mathcal{B}$. We call $\mathcal{B}$ a \emph{Gr\"{o}bner basis} of $\mathcal{L}$ if $\mathcal{F}(u)_{\mathcal{B}}$ has a unique sink for all $u \in \mathbb{N}^n$. Then, $\mathcal{F}(u)_{\mathcal{B}}$ is called a \emph{fiber graph of a Gr\"{o}bner basis}. It is important to note that since our focus is on algebraic statistics and Markov bases, we undirect the edges of the fiber graphs of Gr\"obner bases and discuss ordinary connectivity instead of strong connectivity of directed graphs.

It is fruitful to view the previous notions from the standpoint of commutative algebra as well. A lattice $\mathcal{L} \subset \mathbb{Z}^n$ can be represented by the \emph{lattice ideal} $I_{\mathcal{L}} = \langle p^u - p^v : u, v \in \mathbb{N}^n, u - v \in \mathcal{L} \rangle \subset \mathbb{R}[p_1, ..., p_n]$. $I_{\mathcal{L}}$ is a toric ideal. We can write $b = b^+ - b^-$ with non-negative $b^+$ and $b^-$ for every $b \in \mathcal{L}$. The following result is considered one of the starting points for algebraic statistics: \begin{thm}[The fundamental theorem of Markov bases, \cite{diaconis-sturmfels}]\label[thm]{thm:fundamental}
A subset $\mathcal{B}$ of the lattice $\mathcal{L}$ is a Markov basis if and only if the corresponding set of binomials $\{p^{b^+} - p^{b^-} : b \in \mathcal{B} \}$ generates the lattice ideal $I_{\mathcal{L}}$.
\end{thm}

In the case of two-way contingency tables, the sufficient statistic is the row and columns sums of the tables and the matrix $A$ is chosen correspondingly. Since we consider the case of equal, fixed row and column sums, all tables are in the same fiber. The integer kernel of $A$ has a Markov basis whose cardinality is $2{n \choose 2}^2$, namely $\mathcal{B} = \{\pm(e_{ij} + e_{kl} - e_{il} - e_{kj}) : 1 \leq i < k \leq n, 1 \leq j < l \leq n \}$, where $e_{ij}$ denotes the matrix which has one in the position $(i, j)$ and zeroes elsewhere. This is exemplified in \cite{diaconis-sturmfels}, and for an explicit proof of a more general result which implies it, see \cite{dss}. By \cref{thm:fundamental}, $\{p^{b^+} - p^{b^-} : b \in \mathcal{B} \}$ generates the lattice ideal $I_{\mathcal{L}}$. One can verify that the Markov basis gives a Gr\"{o}bner basis by the cost vector with $(r + c)^2$ for the element on row $r$ and column $c$. Since we need to have $b \cdot w < 0$, the generators need to be of the form $-(e_{ij} + e_{kl} - e_{il} - e_{kj})$. The reason why the corresponding fiber graph has a unique sink is that the moves of this form are not possible from the anti-diagonal contingency table. The fact that this Gr\"{o}bner basis is \emph{reduced} is justified, for example, in Chapter 5 of \cite{gbcp}. As mentioned earlier, for the purposes of this paper, we undirect the edges of the fiber graph of the Gr\"obner basis. This means that the edges in our graph $G(n, r)$ correspond exactly to the elements of the Markov basis $\mathcal{B}$, Markov moves.

\subsection{Basic notation of graph theory}
Next, we define a number of basic notions for graphs following those in \cite{diestel}. Let $G$ be a graph, $V(G)$ be the \emph{vertex set} of $G$ and $|G| = |V(G)|$. The \emph{degree} $d(v)$ of a vertex $v$ in $G$ is the number of edges at $v$. The \emph{minimum degree} $\delta(G)$ of a graph $G$ is the smallest of the degrees in the graph and the \emph{maximum degree} $\Delta(G)$ the largest. We call a graph $G$ $k$\emph{-connected}, $k \in \mathbb{N}$, if $|G| > k$ and $G - X$ is connected for every set $X \subseteq V(G)$ with $|X| < k$. The notation $G - X$ means a graph with the vertex set $V(G) - X$ and edges of $G$ such that their endpoints are in $V(G) - X$. A subgraph of this type is called an \emph{induced subgraph} of $G$. By Menger's Theorem \cite[p.~71]{diestel}, a graph is $k$-connected if and only if it contains $k$ independent (in other words, vertex-disjoint) paths between any two vertices. We will use disjoint as a synonym of independent. The \emph{connectivity} ${\kappa}(G)$ of a graph $G$ is the largest $k$ such that $G$ is $k$-connected, the \emph{distance} $d_G(u,v)$ between two vertices $u$ and $v$ of $G$ is the number of edges in a shortest $u-v$ path in $G$, and the \emph{diameter} diam$(G)$ of $G$ is defined as the largest distance in $G$. The graph $G$ is $r$\emph{-regular} if, all its vertices have the same degree $r$. If $V(G)$ admits a partition into two classes such that the vertices in the same class are not adjacent, $G$ is called a \emph{bipartite} graph. A \emph{matching} $M$ in $G$ is a set of independent edges and it is called \emph{perfect} if every vertex of $G$ is incident to exactly one edge in $M$. A \emph{multigraph} is a pair $(V,E)$ of disjoint sets together with a map $E \mapsto [V]^2$ that assigns two vertices to each edge. Here $E$ denotes the set of edges. A multigraph differs from an ordinary graph by allowing several edges between the same two vertices. As opposed to the definition in \cite{diestel}, our definition does not allow self-loops, edges that start from and end to the same vertex. The entry $a_{ij}$ of the \emph{adjacency matrix} $A$ of a multigraph is the number of edges from the vertex $i$ to the vertex $j$. We define the \emph{biadjacency matrix} of a bipartite multigraph as the submatrix of the adjacency matrix, where the columns correspond to the vertices in a bipartition class of the vertex set and rows to the vertices in the other class.

\section{The fiber graphs}

The first results are on the degree of the vertices of $G(n,r)$. The degree $d(v)$ of $v \in V(G(n,r))$ is exactly the number of Markov moves that can be performed from $v$. From here on, a \emph{move} means a Markov move. Recall that here the set of Markov moves is the Markov basis \[\mathcal{B} = \{\pm(e_{ij} + e_{kl} - e_{il} - e_{kj}) : 1 \leq i < k \leq n, 1 \leq j < l \leq n \}.\] Thus, we want to calculate the number of unordered pairs \[\{v_{i_1j_1}, v_{i_2j_2} : v_{i_1j_1}, v_{i_2j_2} > 0, i_1 \neq i_2, j_1 \neq j_2\}.\] If we have such a pair, the entries $v_{i_2j_1}$ and $v_{i_1j_2}$ cannot be $r$, and the move $e_{i_1j_2} + e_{i_2j_1} - e_{i_1j_1} - e_{i_2j_2}$ must then be possible from $v$. We define the support of a vertex $v \in V(G(n,r))$ as the set \[\text{supp}(v) = \{(i, j) : i, j \in \mathbb{N}, 1 \leq i, j \leq n, v_{ij} \neq 0 \}.\] Note that the cardinality of $\text{supp}(v)$ is the number of positive entries in $v$.

\begin{lem}
\label[lem]{lem:degrees}
Let $v \in V(G(n,r))$,
\begin{enumerate}[{\normalfont (a)}] 
\item if $v$ has $r$ as its only positive entries, then $d(v) = \delta(G(n,r)) = {n \choose 2}$.
\item if $v$ does not have $r$ as its only positive entries, then $d(v) \geq \frac{(n+2)(n-1)}{2} = {n \choose 2} + n - 1$.
\end{enumerate}
\end{lem}
\begin{proof}
\emph{Part a).} Since there are exactly $n$ nonzero entries in $v$, all of them in different rows and columns, there are ${n \choose 2}$ pairs $\{v_{i_1j_1}, v_{i_2j_2} : v_{i_1j_1}, v_{i_2j_2} > 0, i_1 \neq i_2, j_1 \neq j_2\}$. Thus, $d(v) = {n \choose 2}$. To prove that $d(v)$ is in this case the minimum degree, we need to prove part b) of this lemma. 

\emph{Part b).}  Consider starting from a vertex that has $r$ as its only positive entries, and therefore support of size $n$, and using a Markov move to get to $v$. Now, because we must have $r > 1$, the size of the support must grow by at least two in the process. Therefore, the size of $\text{supp}(v)$ is at least $n+2$. The pair $\{v_{i_1j_1}, v_{i_2j_2} : v_{i_1j_1}, v_{i_2j_2} > 0, i_1 \neq i_2, j_1 \neq j_2\}$ can be picked in $\frac{(n+2)(n-1)}{2}$ ways, because the row $i_1$ and column $j_1$ cannot contain any $r$-entries if $v_{i_1j_1}$ is positive, and there are $n-1$ other rows and columns which then need to contain positive entries. Hence, there are at least $\frac{(n+2)(n-1)}{2}$ Markov moves from $v$, and $d(v) \geq \frac{(n+2)(n-1)}{2} = {n \choose 2} + n - 1$.
\end{proof}

Using the definition of connectivity with $X = \{v\}$, $v$ being the vertex with all positive entries equal to $r$, we get the following result as an immediate implication of \cref{lem:degrees}. 

\begin{prop}\label[prop]{prop:connmax}
The connectivity of $G(n,r)$ satisfies $\kappa(G(n,r)) \leq {n \choose 2}$.
\end{prop}

\begin{prop}
If $V(G(n,r))$ contains a vertex $v$ that has one as its only positive entries, then $\Delta(G(n,r)) = d(v) = \frac{nr(nr-2r+1)}{2}$.
\end{prop}

\begin{proof}
If $r = 1$, we are done by \cref{lem:degrees} and the fact that all the vertices need to be of this type. Thus, assume $r > 1$. There has to be $nr$ one-entries. Say that one of them is in the position $(i_1, j_1)$. The row $i_1$ and column $j_1$ contain $2r-1$ other one-entries. Therefore, there are $\frac{nr(nr-2r+1)}{2}$ pairs $\{v_{i_1j_1}, v_{i_2j_2} : v_{i_1j_1}, v_{i_2j_2} > 0, i_1 \neq i_2, j_1 \neq j_2\}$. All of these pairs correspond to a different Markov move from $v$, and thus $d(v) = \frac{nr(nr-2r+1)}{2}$. Because the row and column sums are $r$ and we have an $n \times n$-matrix, $|\text{supp}(u)| \leq nr$ for any $u \in V(G(n,r))$. The equality corresponds to the case where the positive entries are all ones. Thus, if there is at least one >one-entry in $u$, the number of positive entries is less than $nr$. Then, we can pick $u_{i_1j_1}$ in less than $nr$ ways. We claim that $d(u) < d(v)$. If $n = 2$, we are done by \cref{lem:degrees}. Let $n \geq 3$, and start from $v$. Perform a Markov move $e_{i'_1j'_1} + e_{i'_2j'_2} - e_{i'_1j'_2} - e_{i'_2j'_1}$ from $v$ to $u$ so that an entry $u_{i'_1j'_1} > 1$. If we would pick $u_{i_1j_1} = u_{i'_1j'_1}$, the entry $u_{i_2j_2}$ in the pair $\{u_{i_1j_1}, u_{i_2j_2} : u_{i_1j_1}, u_{i_2j_2} > 0, i_1 \neq i_2, j_1 \neq j_2\}$ could be chosen in at most one more way than $v_{i_2j_2}$ for the corresponding $v_{i_1j_1}$, because by our assumption, only $u_{i'_2j'_2}$ can be both positive and such that its position, $(i'_2, j'_2)$, is not in the support of $v$. For any other $u_{i_1j_1}$, the number of pairs $u_{i_2j_2}$ can only decrease or stay the same, since we can assume that $(i'_1, j'_2)$ and $(i'_2, j'_1)$ are in the support of $v$ but not in the support of $u$. Moreover, since $n \geq 3$, $r > 1$ and we assume that $v_{i'_1j'_2} = v_{i'_2j'_1} = 1$, if $v_{i'_2j'_2} = 0$ there is a positive entry of $u$ such that it is in the row $i'_2$ but not in the column $j'_2$. If the pick $u_{i_1j_1}$ is that entry, by our assumption there is one less possible pair $u_{i_2j_2}$ for $u_{i_1j_1}$ than a pair $v_{i_2j_2}$ for the corresponding $v_{i_1j_1}$, because $u_{i'_1j'_2} = 0$, but $v_{i'_1j'_2} = 1$. On the other hand, if $v_{i'_2j'_2} > 0$, the number of pairs $\{u_{i_1j_1}, u_{i_2j_2} : u_{i_1j_1}, u_{i_2j_2} > 0, i_1 \neq i_2, j_1 \neq j_2\}$ where $u_{i_1j_1} = u_{i'_1j'_1}$ does not change while moving from $v$ to $u$. As we iterate the process from $u$, similar arguments hold. Thus, because we could pick $u_{i_1j_1}$ in less than $nr$ ways, $d(u) < d(v)$ and $\Delta(G(n,r)) = \frac{nr(nr-2r+1)}{2}$. Therefore, $d(v) = \Delta(G(n,r)) = \frac{nr(nr-2r+1)}{2}$. 
\end{proof}

Note that when $n < r$, there is no such $v$ with all positive entries equal to one. Nevertheless, the maximum degree obtained is an upper bound for the vertex degree in that case as well. Thus, we know that ${n \choose 2} \leq d(G) \leq \frac{nr(nr-2r+1)}{2}$. Now, having information on how the degree of the vertices of $G(n,r)$ behaves, we try to find the connectivity $\kappa(G(n,r))$. First, we will introduce a couple of auxiliary results:

\begin{lem}\label[lem]{lem:commonchoices}
The number of same Markov moves $M$ from $u, v \in V(G(n,r))$ with $d_G(u,v) \leq 2$ is at least ${n \choose 2}$ for $r > 2$.
\end{lem}
\begin{proof}
Because $d_G(u,v) \leq 2$ and $r > 2$, $|\text{supp}(u)\cap\text{supp}(v)| \geq n$. We want to know whether at least ${n \choose 2}$ pairs of those positions are usable by a Markov move $M$. Those positive entries in $u$ that are not in the support of $v$ must equal 1 or 2. Then, because $r > 2$, there has to be entries $e_i$ satisfying $1 \leq e_i \leq r-1$, at least one in the same column and one in the same row as such an entry. In general, each of the columns not containing an $e_i$ has to contain a positive entry as well. Having a positive entry in a particular column means that there cannot be an $r$-entry in the same row. Thus, there is a positive entry not in this row in each of the other columns. We can choose a pair $\{(i_1, j_1), (i_2, j_2) : i_1 \neq i_2, j_1 \neq j_2\} \subset \text{supp}(u)\cap\text{supp}(v)$ in total in ${n \choose 2}$ ways by first selecting one of the $n$ columns and then one of the $(n-1)$ other columns.
\end{proof}

\begin{thm}[K\H{o}nig, \cite{konig}]\label[thm]{thm:konig}
Every $r$-regular bipartite multigraph decomposes into $r$ perfect matchings.
\end{thm}

Let $E_n(i,j)$ be the $n{\times}n$-matrix with all entries 0, except for that position $(i,j)$ is 1.

\begin{lem}\label[lem]{lem:decomp}
Let $u$ be a vertex of $G(n,r)$ and $(i_1,j_1),...,(i_k,j_k)$ positions in an $n{\times}n$-matrix such that $u \geq E_n(i_1,j_1)+...+E_n(i_k,j_k)$, and $k \leq r$.
Then there is a decomposition of $u$ into a sum of matrices $u_1+...+u_r$ that are vertices of $G(n,1)$ such that
$u_1+...+u_l \geq E_n(i_1,j_1)+...+E_n(i_l,j_l)$ for all $1 \leq l \leq k$.
\end{lem}

\begin{proof}
The proof is by induction on $k$. For $k=0$ we are done. According to \cref{thm:konig}, every $r$-regular bipartite multigraph decomposes into $r$ perfect matchings. Interpreting $u$ as the biadjacency matrix of an $r$-regular bipartite multigraph, we get a decomposition into matrices $u_1+...+u_r$ with row and column sum 1. Assume that we have indexed the matrices such that $(u_1)_{i_1,j_1}>0$. Let $L$ be a maximal subset of $\{1,2,..,k\}$ with 1, such that $u_1 \geq \sum_{l \in L} E_n(i_l,j_l)$. By induction we can find a decomposition of $u-u_1$ admitting the conditions for $\{(i_l,j_l)$ | $l \in \{1,2,..,k\}\setminus L\}$, and then we extend it.
\end{proof}

It might be of interest to the reader that the previous result, \cref{lem:decomp}, implies that the semigroup generated by permutation matrices is a normal cone.

\begin{prop}\label[prop]{prop:conn}
The graph $G(n,r)$ is connected.
\end{prop}

\begin{proof}
The graph $G(n,r)$ is the underlying undirected graph of a fiber graph of a Gr\"{o}bner basis, and therefore connected. 
\end{proof}

\begin{lem}[Liu's criterion, \cite{liu}]\label[lem]{lem:liuc}
 Let $G$ be a connected graph and
$|V(G)| > k$. If for any two vertices $u$ and $v$ of $G$ with distance
$d_G(u,v) = 2$ there are $k$ disjoint $u-v$ paths in $G$, then $G$ is $k$-connected.
\end{lem}

A proof of \cref{lem:liuc} can be found in \cite{bjorner-vorwerk}. With these tools, we can set out to prove our main result:

\begin{thm}\label[thm]{thm:maintheorem}
The connectivity $\kappa(G(n,r)) = {n \choose 2}$ for $r > 2$.
\end{thm}

\begin{proof}

By \cref{prop:connmax}, $\kappa(G(n,r)) \leq {n \choose 2}$. Therefore, our goal is to show that $G(n,r)$ is ${n \choose 2}$-connected. We aim to achieve this by applying \cref{prop:conn} and \cref{lem:liuc} as well as a technique of building a large number of paths. We need to show that using the technique, we will in every case get at least ${n \choose 2}$ independent paths. It turns out that the technique used will not work in the cases $r < 3$. If $n = 2$, ${n \choose 2} = 1$. By \cref{prop:conn}, $G(n,r)$ is connected and the case $n = 2$ is done. Thus, we assume from now on that $n \geq 3$.

We will start by setting up the machinery. By \cref{prop:conn}, we can apply \cref{lem:liuc}. Let $u, v$ $\in V(G(n,r))$ with $d_G(u,v) = 2$. Then there are Markov moves $\Delta_1$ and $\Delta_2$ such that $u + \Delta_1 + \Delta_2 = v$. Because $d_G(u,v) = 2$, $\Delta_1 + \Delta_2$ does not correspond to a single move. Now, let us consider the sequences $M, \Delta_1, \Delta_2, -M$, where $M$ is an additional move, such that $u + M + \Delta_1 + \Delta_2 - M = v$, as depicted in \cref{fig:idea}. Let $c_M$ be the number of ways to select $M$ so that we get disjoint paths. We want to show that $c_M$ is at least ${n \choose 2} - 1$. Then we would have in total ${n \choose 2}$ disjoint paths between $u$ and $v$ when we count the original path of length two as well. Note that the move $M$ has to be a valid Markov move from $u$. By \emph{valid}, we mean that the move does not take entries of $u$ negative (or correspondingly, larger than $r$). In other words, the move $M$ needs to connect $u$ to another vertex in the graph. The term \emph{possible} move is used as a synonym for valid move.

\begin{figure}[htbp!]
\centering
\includegraphics[scale=1.0]{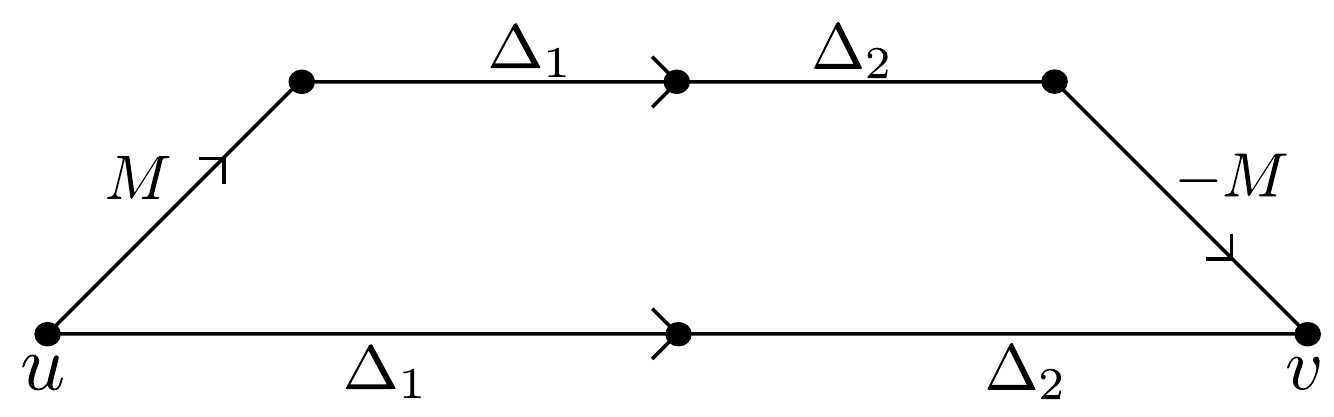}
\caption{The types of paths considered in the proof with the directions corresponding to the signs of the moves.}
\label[fig]{fig:idea}
\end{figure}

There are some remarks to be made:
\begin{itemize}
\item We must have $M \neq \Delta_1$, because $M = \Delta_1$ would lead to an intersection. For the same reason, we need $M \neq -\Delta_2$. 
\item If we can use $M = \Delta_2$ or $M = -\Delta_1$, we have $u + \Delta_2 + \Delta_1 = v$. However, if both of them are valid Markov moves from $u$, we have to subtract one from $c_M$, because then the paths with $M = \Delta_2$ and $M = -\Delta_1$ intersect.
\item On the other hand, if the move $M = \Delta_2$ is not valid, the path using $M = -\Delta_1$ does not connect $u$ and $v$.
\item If $r = 1$, the entries $\Delta_1$ subtracts from are not usable by $M$. By \cref{lem:degrees}, in that case each of the vertices have the degree ${n \choose 2}$, and thus this method does not apply, because we will not get enough ways of choosing $M$. 
\item If $r = 2, M, \Delta_1$ and $\Delta_2$ cannot have even one same entry where they subtract from, again problematic in the cases where we start from a vertex with the degree ${n \choose 2}$. Then we cannot get the desired result using solely this procedure. For simplicity, assume $r \geq 3$.  
\end{itemize}

\begin{component}[The basic case]
Let us first assume that $M$ can subtract from the same entries as $\Delta_1$ and $\Delta_2$. By this, we mean that the entries are large enough that we do not have to worry whether using $M$ before $\Delta_1$ and $\Delta_2$ causes an entry to be negative after performing $M + \Delta_1$ or $M + \Delta_1 + \Delta_2$. Consider $\Delta_1 = \Delta_2$.
\begin{itemize}
\item We have at least ${n \choose 2} - 1$ ways of choosing $M$ such that $M \neq \Delta_1$, because d$(u) \geq {n \choose 2}$ by \cref{lem:degrees}. 
\item If it is even possible to select $M = -\Delta_1$, some nonzero-entries of $u$ are not $r$, and by \cref{lem:degrees}, the degree of the vertex we are at is at least ${n \choose 2} + n - 1$. Therefore, after subtracting the disallowed moves $M = \Delta_1$ and $M = -\Delta_1$, we have $c_M \geq {n \choose 2} + n - 3 \geq {n \choose 2}$ in this case, because $n \geq 3$.
\end{itemize}
However, we also have to take the case $\Delta_1 \neq \Delta_2$ into account.
\begin{itemize} 
\item If $M = \Delta_2$ is possible, but $d(u) = {n \choose 2}$, $M = -\Delta_2$ is not possible. Therefore, the previous results hold in this case as well.
\item If also $M = -\Delta_1$ is possible as well as $M = -\Delta_2$, by the earlier analysis we get $c_M \geq {n \choose 2} - 1$, because with our assumption $n \geq 3$, ${n \choose 2} + n - 4 \geq {n \choose 2} - 1$. 
\item If on the other hand $M = \Delta_2$ is not possible, we want to know whether the possibility $M = -\Delta_2$ is included in $d(u)$. If the number of entries of $\Delta_2$ that prevent its use is at least two, $\Delta_1$ must be $-\Delta_2$, because $\Delta_2$ will then subtract from entries zero in $u$ $\Delta_1$ adds to. However, there is no point in this. Therefore, consider that $\Delta_2$ has only one entry obstructing its use. Then $\Delta_2$ subtracts from an entry zero in $u$, which implies that $\Delta_1$ has to add to that entry, but then $-\Delta_1$ would subtract from the entry. Thus, $M = -\Delta_1$ is not included in $d(u)$. If $M = -\Delta_2$ is to be possible, by \cref{lem:degrees}, we need to be at $u$ with $d(u) \geq {n \choose 2} + n - 1$, because otherwise we would subtract from an $r$-entry with $\Delta_2$, but then we would add to a zero-entry. Then $c_M \geq {n \choose 2} + n - 3 \geq {n \choose 2}$. Otherwise we only need to avoid $M = \Delta_1$ and have $c_M \geq {n \choose 2} - 1$, because $d(u) \geq {n \choose 2}$ by \cref{lem:degrees}.
\end{itemize}
\end{component}

\begin{component}[Problematic entries]
Let us now move on to the cases where $M$ cannot subtract from all the entries where $\Delta_1$ and $\Delta_2$. Then, the moves $\Delta_1$ and $\Delta_2$ subtract from entries smaller than two in $u$. The number of this kind of \emph{problematic entries} can range from one to four. By \cref{lem:commonchoices}, the number of same Markov moves $M$ from $u$ and $v$ must be at least ${n \choose 2}$.
\begin{itemize}
\item First, say that $\Delta_1 + \Delta_2$ subtracts from either four or three one-entries or two or one two-entry. Then the choices of moves at $v$ do not include $\Delta_1$ or $\Delta_2$. We have to avoid $-\Delta_2$, and thus $c_M \geq {n \choose 2} - 1$. 
\item If $\Delta_1$ and $\Delta_2$ subtract from three one-entries in total, but the sum $\Delta_1 + \Delta_2$ does not, there are six different cases: either $\Delta_1$ or $\Delta_2$ subtracts from two one-entries, and $\Delta_1$, $\Delta_2$ or both add to an entry the other subtracts from. If $\Delta_1$ subtracts from two one-entries, the moves $\Delta_1$ and $\Delta_2$ are clearly not possible at $v$. Then we have $c_M \geq {n \choose 2} - 1$. The same thing happens when $\Delta_2$ subtracts from two one-entries and $\Delta_2$ does not add to an entry $\Delta_1$ subtracts from. In the two cases left, we cannot rely on \cref{lem:commonchoices}. 
\item If $\Delta_1$ and $\Delta_2$ subtract from a total number of two one-entries, we either have the other one subtracting from two or both subtracting from one. In the latter case, if neither of them or only $\Delta_1$ adds to an entry the other subtracts from, $\Delta_1$ and $\Delta_2$ are not possible at $v$. Hence, in this case as well, we have $c_M \geq {n \choose 2} - 1$. 
\end{itemize}
The cases left are: only $\Delta_1$ or $\Delta_2$ subtracts from one-entries; $\Delta_1$ subtracts from one one-entry, while $\Delta_2$ subtracts from at least one different one-entry but adds to the one-entry $\Delta_1$ subtracts from. 
\begin{itemize}
\item If $\Delta_1$ subtracts from one one-entry, $u \geq E_n(i,j)$ where $(i,j)$ is the position of that particular one-entry. Following \cref{lem:decomp}, decompose $u$: $u = u_1 + ... + u_r = u_1 + u'$, where $u_1 \geq E_n(i,j)$ and $u' \in V(G(n,r-1))$. The one-entry in the position $(i,j)$ in $u$ is now zero in $u'$. Because $u$ has one one-entry, it must have at least another. The second one-entry can either be in $u_1$ or $u'$. If it is in $u_1$, $d(u') \geq {n \choose 2}$, and if it is in $u'$, $d(u') \geq {n \choose 2} + n - 1$ by \cref{lem:degrees}. In the former case we get $c_M \geq {n \choose 2} + (n-2) - 2 \geq {n \choose 2} - 1$, where ${n \choose 2}$ comes from the moves for $u'$ and $(n-2)$ from the moves using the one entry not problematic in $u$ now in $u_1$. In the latter case we have $c_M \geq = {n \choose 2} + (n-1) - 2 \geq {n \choose 2}$. We subtract two in both cases to avoid counting $M = -\Delta_1$ and $M = -\Delta_2$. 
\item If $\Delta_1$ subtracts from two one-entries at positions $(i_1,j_1)$ and $(i_2, j_2)$, $u \geq E_n(i_1, j_1) + E_n(i_2,j_2)$, and we decompose $u = u_1 + u_2 + u'$, where $u' \in V(G(n, r-2))$ and $u_1 + u_2 \geq E_n(i_1, j_1) + E_n(i_2, j_2)$. Thus, the problematic entries are zero in $u'$, and therefore also the move $\Delta_1$ is not possible from $u'$. If $d(u') = {n \choose 2}$, $-\Delta_1$ is not included in $d(u')$, and we have $c_M = {n \choose 2} - 1$. Otherwise $d(u') > {n \choose 2}$, and we get $c_M \geq {n \choose 2} - 1$.
\item If $\Delta_2$ subtracts from one-entries some of which are also in $u$, the case is treated exactly the same way as the two previous ones. If the particular one-entries are not in $u$, $M$ cannot subtract from them and thus there is nothing to avoid. 
\item The case where $\Delta_1$ subtracts from one one-entry and $\Delta_2$ subtracts from one or two different one-entries, but $\Delta_2$ adds to the one-entry $\Delta_1$ subtracts from and at most one of the one-entries $\Delta_2$ subtracts from is present in $u$ already, is treated exactly same way as the previous ones, because we have to avoid one or two problematic one-entries. If there are two problematic entries both already in $u$, they can be avoided the same way as before. If there are three of them, all present in $u$ at positions $(i_1,j_1)$, $(i_2, j_2)$ and $(i_3, j_3)$, we have $u \geq E_n(i_1, j_1) + E_n(i_2,j_2) + E_n(i_3,j_3)$. Say that the two first are the ones used by $\Delta_2$. They can be put in the same $u_1$ in the proof of \cref{lem:decomp}. Then we have $u = u_1 + u_2 + u'$, where $u' \in V(G(n, r-2))$. The problematic entries are zero in $u'$ and the moves $\Delta_1$ and $\Delta_2$ are not possible from $u'$. Then $d(u') \geq {n \choose 2}$ only includes the disallowed choice $M=-\Delta_2$. Thus $c_M \geq {n \choose 2} - 1$.
\end{itemize}
\end{component}

\begin{component}[Intersections]
The last question is what if different paths $M + \Delta_1 + \Delta_2 - M$ and $M' + \Delta_1 + \Delta_2 - M'$ intersect. By symmetry and straightforward calculations, the number of cases reduces to three: $M' - M = \Delta_1$; $M' - M = \Delta_2$; $M' - M = \Delta_1 + \Delta_2$. The different types are drawn in \cref{fig:intersections}:

\begin{figure}[htbp!]
\centering
\includegraphics[scale=1.0]{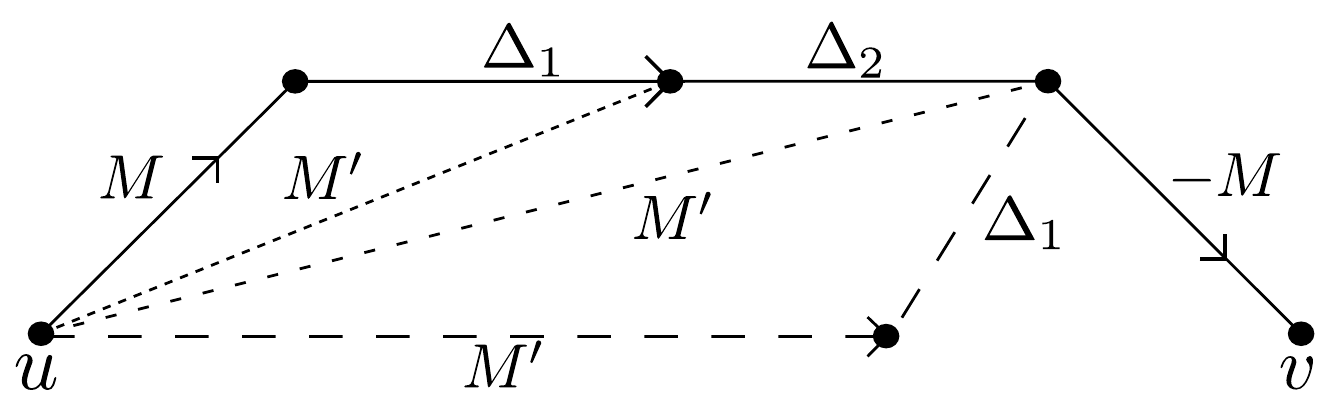}
\caption{The possible types of intersection.}
\label[fig]{fig:intersections}
\end{figure}

The last case is the easiest to handle. Assume that we only have intersections of this type. An intersection can happen in two different ways. 
\begin{itemize}
\item The moves $\Delta_1$ and $\Delta_2$ share one entry the other adds to and the other subtracts from. This sum can be written in three ways, one being the original, because $M$ and $M'$ have to be Markov moves, both have to use three of the operations in $\Delta_1 + \Delta_2$ and the cancelling operations can be done to three different entries. Therefore, this case amounts to two intersections. Let us a write an example to illustrate this: \begin{equation*} \begin{split}
\Delta_1 + \Delta_2 &= \left(
\begin{array}{rrr}
 	1 & -1 & 0 \\
	-1 & 1 & 0 \\
	0 & 0 & 0 \\
\end{array}
\right) + \left(\begin{array}{rrr}
 	0 & 0 & 0 \\
	0 & -1 & 1 \\
	0 & 1 & -1 \\
\end{array}
\right)\\  &= \left(
\begin{array}{rrr}
 	0 & 0 & 0 \\
	-1 & 0 & 1 \\
	1 & 0 & -1 \\
\end{array}
\right) + \left(\begin{array}{rrr}
 	1 & -1 & 0 \\
	0 & 0 & 0 \\
	-1 & 1 & 0 \\
\end{array}
\right)\\ &= \left(
\begin{array}{rrr}
 	1 & 0 & -1 \\
	-1 & 0 & 1 \\
	0 & 0 & 0 \\
\end{array}
\right) + \left(\begin{array}{rrr}
 	0 & -1 & 1 \\
	0 & 0 & 0 \\
	0 & 1 & -1 \\
\end{array}
\right) = \left(
\begin{array}{rrr}
 	1 & -1 & 0 \\
	-1 & 0 & 1 \\
	0 & 1 & -1 \\
\end{array}
\right). 
\end{split}
\end{equation*} 
\item The other possibility, disjoint from the previous one, is that the positions of the non-zero rows or columns of $\Delta_1$ and $\Delta_2$ are the same. Then there are two ways, the original and another with swapped rows, to write the sum $\Delta_1 + \Delta_2$. This gives one intersection. Again, let us do a basic example:
\begin{equation*} \begin{split} 
\Delta_1 + \Delta_2 &= \left(
\begin{array}{rrrr}
 	1 & -1 & 0 & 0 \\
	-1 & 1 & 0 & 0 \\
	0 & 0 & 0 & 0 \\
	0 & 0 & 0 & 0 \\
\end{array}
\right) + \left(\begin{array}{rrrr}
 	0 & 0 & 0 & 0 \\
	0 & 0 & 0 & 0 \\
	-1 & 1 & 0 & 0 \\
	1 & -1 & 0 & 0 \\
\end{array}
\right)\\ &= \left(
\begin{array}{rrrr}
 	0 & 0 & 0 & 0 \\
	-1 & 1 & 0 & 0 \\
	0 & 0 & 0 & 0 \\
	1 & -1 & 0 & 0 \\
\end{array}
\right) + \left(\begin{array}{rrrr}
 	1 & -1 & 0 & 0 \\
	0 & 0 & 0 & 0 \\
	-1 & 1 & 0 & 0 \\
	0 & 0 & 0 & 0 \\
\end{array}
\right).
\end{split} 
\end{equation*}
\end{itemize} 

To analyse how these affect the earlier calculations, we have to first note that the entries $\Delta_1 + \Delta_2$ subtracts from must be at least two, because we want to subtract from the same entries with $M$ and $M'$. Only one of the two types is possible at a time. 
\begin{itemize}
\item In the former case, there are two possibilities: either one or two intersections are possible. Let us first consider the case of one intersection. If $-\Delta_2$ (or $-\Delta_1$ if the order of the moves is switched) is to be included in $d(u)$, we must have $d(u) > {n \choose 2} + n - 1$, because there has to be at least three positive entries in one column, and therefore $c_M > {n \choose 2} + n - 1 - 4 \geq {n \choose 2} - 2$. The $-4$ comes from three disallowed moves and one intersection. If not, the degree is at least ${n \choose 2} + n - 1$ by \cref{lem:degrees}, which means we have $c_M \geq {n \choose 2} + n - 4 \geq {n \choose 2} - 1$. Now, assume that there are two intersections. The sum $\Delta_1 + \Delta_2$ shows that there must be at least three positive entries in the $3\times3$-submatrix. However, when we write the sum in another way, the other move is $-M$. Thus, there has to be three additional positive entries, because in the different cases, $M$ subtracts in total from at least two of the one-entries in $\Delta_1 + \Delta_2$ and one other entry. Hence, the support of $u$ has size at least $n + 3$. If we pick a positive entry from the $3\times3$-submatrix to be subtracted from by a Markov move, and the entry is such that four of the other five positive entries are on its row or column, the selection of the second positive entry can be done in $n - 1$ ways, because there cannot be $r$-entries in the row or column of the first entry. Clearly, if we pick the first entry in a different way, there are cases where the second selection can be done in even more ways, but no cases where in less. Thus, $c_M > \frac{(n + 3)(n - 1)}{2} - 5 \geq {n \choose 2} - 2$, because $n \geq 3$. The $-5$ comes from three disallowed moves and two intersections.
\item In the latter case, $d(u)$ must be at least ${n \choose 2} + n - 1$ by \cref{lem:degrees}, and we have $c_M \geq {n \choose 2} + n - 1 - 4 \geq {n \choose 2} - 1$, because we must have $n \geq 4$. We subtract four, because there are at most three disallowed moves and one intersection.
\end{itemize}

In the two other cases we have $M'$ and $M$ sharing one non-zero row or column, which disappears in the sum $M'+(-M)$. An example is presented below: \begin{equation*} \begin{split} 
\Delta &= \left(
\begin{array}{rrrr}
 	1 & -1 & 0 & 0 \\
	-1 & 1 & 0 & 0 \\
	0 & 0 & 0 & 0 \\
	0 & 0 & 0 & 0 \\
\end{array}
\right) = \left(
\begin{array}{rrrr}
 	0 & 0 & 0 & 0 \\
	-1 & 1 & 0 & 0 \\
	0 & 0 & 0 & 0 \\
	1 & -1 & 0 & 0 \\
\end{array}
\right) + \left(\begin{array}{rrrr}
 	1 & -1 & 0 & 0 \\
	0 & 0 & 0 & 0 \\
	0 & 0 & 0 & 0 \\
	-1 & 1 & 0 & 0 \\
\end{array}
\right)\\ &= \left(
\begin{array}{rrrr}
 	1 & 0 & -1 & 0 \\
	-1 & 0 & 1 & 0 \\
	0 & 0 & 0 & 0 \\
	0 & 0 & 0 & 0 \\
\end{array}
\right) + \left(\begin{array}{rrrr}
 	0 & -1 & 1 & 0 \\
	0 & 1 & -1 & 0 \\
	0 & 0 & 0 & 0 \\
	0 & 0 & 0 & 0 \\
\end{array}
\right) = M'+(-M). 
\end{split}
\end{equation*}  
We assume that either $\Delta_1$ or $\Delta_2$ causes intersections, and denote the one causing them with $\Delta$. Let the other one be $\Delta'$. Because $M'$ and $-M$ share one row with $\Delta$, $\Delta$ also adds to a positive entry, because $M$ needs to subtract from that. Let the position of that entry be $(i_1, j_1)$. 

\begin{itemize}
\item Assume that at least one of the entries $e_i$ $\Delta'$ subtracts from satisfies $1 \leq e_i \leq r-1$. Then there must be at least one positive entry in the same column and one in the same row. If they are both in the row $i_1$ and column $j_1$, $e_i$ is in the position $(i_1, j_1)$. Otherwise, we can find 1-entries that do not use the row $i_1$ and the column $j_1$ for each $e_i$ $\Delta'$ subtracts from satisfying $1 \leq e_i \leq r-1$. Denote them with $(i_2, j_2)$ and $(i_3, j_3)$. It might be that $(i_3, j_3)$ does not exist or $(i_2, j_2) = (i_3, j_3)$. We have $u \geq E_n(i_1, j_1) + E_n(i_2, j_2) + E_n(i_3, j_3)$. They can all be put in the same $u_1 \in V(G(n, 1))$ in the construction of the proof of \cref{lem:decomp}. Thus, by \cref{lem:decomp}, we have $u = u_1 + u'$, where $u_1$ is such that it does not contain the entries at most $r-1$ $\Delta$ or $\Delta'$ subtract from. Because $d(u_1) = {n \choose 2}$, and the choices $M = \Delta_1$ and $M = \Delta_2$ as well as intersections are avoided in $u_1$, we have $c_M \geq {n \choose 2} - 1$.
\item Now, assume that both of the entries $\Delta'$ subtracts from are $r$. As before, decompose $u = u_1 + u'$ using \cref{lem:decomp}. This time, we cannot avoid the entries used by $\Delta'$, but they will surely be large enough to be usable by $M$. Again, $u_1$ does not contain the entries subtracted from by $\Delta$. Hence, we cannot have intersections of the other type occuring with moves from $u_1$ and have to only avoid $M = -\Delta_1$, because $\Delta_2$ adds to zero-entries, and thus $M = -\Delta_2$ is not included in $d(u_1)$. We have $c_M \geq {n \choose 2} - 1$.
\end{itemize}
In the latter case, $\Delta'$ cannot cause intersections because of the assumption that $\Delta'$ subtracts from $r$-entries, but in the former case it could. Because the entries $\Delta'$ subtracts from are in $u'$, the calculations hold even if intersections of the type $M' - M = \Delta'$ are assumed possible. \qedhere
\end{component}
\end{proof}

The last result in this paper concerns the diameter of $G(n,r)$:

\begin{prop}
The diameter of $G(n,r)$ is $(n-1)r$.
\end{prop}

\begin{proof}
Every row sum is $r$, and each of the positive entries can be selected to be subtracted from. Therefore, $r$ changes are enough to transform a row to any other. The $n$:th row must be correct at least after changing the $(n-1)$:th row, because otherwise we would have to change an already correct row to incorrect. The maximal number of changes needed is then $(n-1)r$, and diam$(G(n,r)) \leq (n-1)r$.  

Now, it suffices to show that diam$(G(n,r)) \geq (n-1)r$. Take the diagonal matrix 

\begin{equation*}
A = \left(
\begin{array}{ccccc}
 	r & 0 & \cdots & 0 \\
	0 & r & \cdots & 0 \\
	\vdots & \vdots & \ddots & \vdots \\
	0 & 0 & \cdots & r \\
\end{array}
\right).
\end{equation*}

The coordinates of the nonzero-entries are of the form $(i,i), i \in \mathbb{N} \cap [1,n]$. Consider permuting the rows so that $(i,i) \mapsto (i,i-1), i \neq 1$, and $(1,1) \mapsto (1,n)$. The result is \begin{equation*} 
A' = \left(
\begin{array}{ccccc}
 	0 & 0 & \cdots & r \\
	r & 0 & \cdots & 0 \\
	\vdots & \vdots & \ddots & \vdots \\
	0 & 0 & \cdots & 0 \\
\end{array}
\right),
\end{equation*} and the permutation matrix \begin{equation*} P = \left(
\begin{array}{ccccc}
 	0 & 0 & \cdots & 1 \\
	1 & 0 & \cdots & 0 \\
	\vdots & \vdots & \ddots & \vdots \\
	0 & 0 & \cdots & 0 \\
\end{array}
\right).\end{equation*} On the other hand, $A = rI$. If $p$ is the number of operations needed to change $\frac{1}{r}A' = P$ to $I$, the number of operations needed to change $A'$ to $A$ is clearly $pr$. 

Consider this procedure: start from the row $i = 1$. Find the row which has its one-entry in the column $i$, in this case the second row, and swap the rows. Repeat this for each of the rows except for the $n$:th one. Before the $(n-1)$:th row is swapped for the second time, it will have its one-entry in the $n$:th column, so by interchanging it with the $n$:th row we will get to $I$.

In our procedure, each of the swaps corrects the place of one one-entry except for the last one which corrects two. However, we might be able to use more swaps that correct two positions. These kind of interchanges require pairs of one-entries to be in positions of the form $(i,j)$ and $(j,i)$. Say that we swap $(i, j)$ with $(i',j')$ to get $(i',j)$ and $(i,j')$. Assume $i > i'$. If $i' < j$ and $i < j'$, $j' > i'$. Thus, the number of entries in a position of the form $(j,i), i > j$ increases by at most one with each swap. There are $n-1$ positive entries in positions of the form $(i,j), i > j$ in $P$. To interchange the positions of two of $n-2$ entries (the entries not in the positions $(1,n)$ and some other) correcting both, we would then need at least one extra swap. Thus, the best possible result we could get this way is still $n-1$ swaps.

Each swap consists of one operation. Thus, $p = n-1$, and therefore we need $(n-1)r$ operations to make $A'$ from $A$. Hence, diam$(G(n,r)) \geq (n-1)r$, but because also diam$(G(n,r)) \leq (n-1)r$, diam$(G(n,r)) = (n-1)r$.
\end{proof}

\section*{Appendix}

In this appendix, we state the technical version of the conjecture mentioned in the introduction. The vertices of a fiber graph are the monomials in the preimage of some monomial $m$ in $\Bbbk[y_1,\ldots,y_n].$ For some fixed lattice ideal and Gr\"obner basis, a fiber graph is $N$-\emph{large} if it is the preimage of a monomial $m$ that is divisible by $(y_1\cdots y_n)^{N}.$ For ideals from contingency tables this corresponds to that each row and column sum is at least $N.$

\begin{conj}[Engstr\"{o}m '12, \cite{e}\cite{p}]
For any lattice ideal with a Gr\"obner basis, there is an $N$ such that the connectivity of any $N$-large fiber graph is given by its minimum vertex degree.
\end{conj}

This is an example by Raymond Hemmecke why the technical condition is needed. Construct a lattice ideal from the $(2k+1)\times (4k+2)$ matrix
\[
\left( 
\begin{array}{ccccccccccccc}
1 && 1 &&& &&&&& -1 & \\
& \ddots && \ddots &&&&&&& \vdots \\
&& 1 && 1&&&&&&-1 \\
&&&&&1 && 1 &&&& -1 & \\
&&&&&& \ddots && \ddots &&& \vdots \\
&&&&&&& 1 && 1&& -1 \\
&&&&&&&&&& 1 & 1 \\
\end{array}
\right)
\]
defining a map $k[x_1,\ldots,x_{4k+2}] \rightarrow k[y_1,\ldots,y_{2k+1}]$ and use the Gr\"obner basis from lexicographic ordering. Then the fiber graph of the preimage of $y_{2k+1}$ is the one dimensional skeleton of two $k$-dimensional cubes connected by one edge. This fiber graph has minimum degree $k$ but is not 2-connected.

\begin{flushleft}
Samu Potka\\
Aalto University\\
Department of Mathematics and Systems Analysis\\
PO Box 11100\\
FI-00076 Aalto\\
Finland\\
e-mail: samu.potka@aalto.fi
\end{flushleft}

\end{document}